\documentclass{amsart}
\usepackage{amsmath,amssymb,amsthm}
\usepackage[english]{babel}
\usepackage{enumerate,bbm}
\newcommand{\eps}{\varepsilon}

\newcommand{\R}{\mathbb R}
\newcommand{\CC}{\mathcal C}
\newcommand{\abs}[1]{\left\vert#1\right\vert}
\newcommand{\norm}[1]{\left\lVert#1\right\rVert}
\newcommand{\set}[1]{\left\{#1\right\}}
\newcommand{\ex}[1]{\mathsf{E}\left[#1\right]}

\newcommand{\ind}[1]{\mathbbm{1}_{#1}}
\newcommand{\oZ}{\overline{Z}}
\newcommand{\oX}{\overline{X}}

\numberwithin{equation}{section} \theoremstyle{plain}
\newtheorem{theorem}{Theorem}[section]

\newtheorem{lemma}[theorem]{Lemma}
\newtheorem{proposition}[theorem]{Proposition}
\theoremstyle{definition}

\newtheorem{remark}[theorem]{Remark}

\begin{document}
\selectlanguage{english}

\title{Mixed stochastic delay differential equations}


\author{Georgiy Shevchenko}
\address{Department of Probability Theory, Statistics and Actuarial Mathematics,
Taras Shevchenko National University of Kyiv,
64 Volodymyrska, 01601 Kyiv, Ukraine}
\email{zhora@univ.kiev.ua}

\begin{abstract}
We consider a stochastic delay differential equation driven by a H\"older continuous process $Z$ and a Wiener process. Under fairly general assumptions on coefficients of the equation, we prove that it has a unique solution. We also give s sufficient condition for finiteness of  moments of the solution and prove that the solution depends on the driver $Z$ continuously.
\end{abstract}

\maketitle

\section{Introduction}

This paper is devoted to a stochastic differential equation of the form
 \begin{equation*}
 X(t) =X(0) +\int_0^t a(s,X)ds+
\int_0^tb(s,X)dW(s)+\int_0^tc(s,X)dZ(s),
 \end{equation*}
where $W$ is a Wiener process, $Z$ is a H\"older continuous  process with H\"older exponent greater than $1/2$, the coefficients $a,b,c$ depend on the past of the process $X$. The integral with respect to $W$ is understood in the usual It\^o sense, while the one with respect to $Z$ is understood in the pathwise sense. (A precise definition of all objects is given in Section~\ref{sec:prelim}.)  We will call this equation a \textit{mixed stochastic delay differential equation}; the word \textit{mixed} refers to the mixed nature of noise, while the word \textit{delay} is due to dependence of the coefficients on the past.
 
In the pure Wiener case, where $c=0$, this equation was considered by many authors, often by the name ``stochastic functional differential equation''. For overview of their results we refer a reader to \cite{mao, mohammed}, where also the importance of such equations is explained, and several particular results arising in applications are given.

In the pure ``fractional'' case, where $b=0$, there are only few results devoted to such equations, considering usually the case where $Z=B^H$ is a fractional Brownian motion (for us, it is also the most important example of the driver $Z$). In \cite{ferrante-rovira,ferrante-rovira1}, the existence of a solution is shown for the coefficients of the form $a(t,X)=a(X(t))$, $b(t,X)= b(X(t-r))$, and $H>1/2$. It is also proved that the solution has a smooth density, and the convergence of solutions is established for a vanishing delay. A similar equation constrained to stay non-negative is considered in \cite{besalu-rovira}. Existence and uniqueness of solution for an equation with general coefficients, also in the case $H>1/2$, are established in \cite{boufoussi-hajji1, leon-tindel}. For such equation, it is proved in \cite{leon-tindel} that the solution possesses infinitely differentiable density, and in \cite{schmalfuss}, that the solution generates a continuous random dynamical system. In \cite{tindel-neuenkirch-nourdin}, 
the unique solvability is established for an equation with $H>1/3$ and coefficients of the form $f(X(t),X(t-r_1),X(t-r_2),\dots)$. 

Concerning mixed stochastic delay differential equations, there are no results known to author. There are some literature devoted to mixed equations without delay. The existence and uniqueness were proved, under different conditions, in \cite{guerra-nualart,kubilius,mbfbm-sde,mbfbm-limit,mbfbm-jumps}. Integrability and convergence results for mixed equations were established in \cite{mbfbm-limit,mbfbm-integr,mbfbm-jumps,mbfbm-malliavin}, and Malliavin regularity was proved in \cite{mbfbm-malliavin}.

In this  paper we show that a mixed stochastic delay differential equation has  a unique solution under rather general assumptions about coefficients. We also provide a condition for the solution to have finite moments of all orders, and a result on the continuity of the solution with respect to the driver $Z$. The latter result allows, in particular, to approximate the solution to a mixed stochastic delay differential equation by solutions to usual stochastic delay differential equations having a random drift.

\section{Preliminaries}\label{sec:prelim}

Let $\bigl(\Omega, \mathcal{F}, \mathbb{F} = \{\mathcal{F}_t, t\ge 0\},
\mathsf{P}\bigr)$ be a complete filtered probability space satisfying the usual assumptions.  

First we fix some notation: throughout the article, $\abs{\cdot}$ will denote the absolute value of a real number, the Euclidean norm of a vector, or the operator norm of a matrix. The symbol $C$ will denote a generic constant, whose
value may change from one line to another. To emphasize its dependence on some parameters, we will put them into subscripts.

We need some notation in order to introduce the main object. For a fixed $r>0$, let $\CC = C([-r,0];\R^d)$ be the Banach space of continuous $\R^d$-valued functions defined on the interval $[-r,0]$ endowed with the supremum norm $\norm{\cdot}_{\CC}$. For a stochastic process $\xi = \set{\xi(t),t\in[-r,T]}$ and $t\in[0,T]$ define a \textit{segment} $\xi_t\in \CC$ by $\xi_t(s) = \xi(t+s)$, $s\in[-r,0]$. Let $a\colon [0,T]\times \CC\to\R^d$, $b_i\colon [0,T]\times \CC\to\R^d$, $i=1,\dots,m$, $c_j\colon  [0,T]\times \CC\to\R^d$, $j=1,\dots,l$, be measurable functions, $Z = \set{Z(t),t\in[0,T]}$ be an $\mathbb{F}$-adapted process in $\R^l$ such that its trajectories are almost surely H\"older continuous of order $\gamma>1/2$. Let also $\eta\colon[-r,0]\to \R^d$ be a $\theta$-H\"older continuous function with $\theta>1-\gamma$. 

Our main object is the following stochastic delay differential equation in $\R^d$:
\begin{equation}\label{sdde-coord}
X(t) = X(0) + \int_{0}^{t} a(s,X_s)ds + \sum_{i=1}^{m} \int_0^t b_i(s,X_s)dW_i(s) + \sum_{j=1}^{l}\int_0^t c_j(s,X_s)dZ_j(s),\quad t\in[0,T],
\end{equation}
with the ``initial condition'' $X(s) = \eta(s)$, $s\in[-r,0]$. In the rest of the paper a shorter notation will be used for equation \eqref{sdde-coord} and its ingredients:
\begin{equation}\label{main-sdde}
X(t) = X(0) + \int_{0}^{t} a(s,X_s)ds +  \int_0^t b(s,X_s)dW(s) + \int_0^t c(s,X_s)dZ(s).
\end{equation}
We remark that it is possible to consider an equation with coefficients depending on the whole past of the process $X$. This can be achieved by just taking $r=T$

The integral with respect to $W$ in \eqref{main-sdde} will be understood in the It\^o sense. The integral with respect to $Z$ is a generalized Lebesgue--Stieltjes integral, defined as follows \cite{Zah98a}. 
For $\alpha\in(0,1)$, define the fractional derivatives
\begin{gather*}
\big(D_{a+}^{\alpha}f\big)(x)=\frac{1}{\Gamma(1-\alpha)}\bigg(\frac{f(x)}{(x-a)^\alpha}+\alpha
\int_{a}^x\frac{f(x)-f(u)}{(x-u)^{1+\alpha}}du\bigg),\\
\big(D_{b-}^{1-\alpha}g\big)(x)=\frac{e^{-i\pi
\alpha}}{\Gamma(\alpha)}\bigg(\frac{g(x)}{(b-x)^{1-\alpha}}+(1-\alpha)
\int_{x}^b\frac{g(x)-g(u)}{(u-x)^{2-\alpha}}du\bigg).
\end{gather*}
Assuming that $D_{a+}^{\alpha}f\in L_1[a,b], \ D_{b-}^{1-\alpha}g_{b-}\in
L_\infty[a,b]$, where $g_{b-}(x) = g(x) - g(b)$, the generalized (fractional) Lebesgue-Stieltjes integral $\int_a^bf(x)dg(x)$ is defined as
\begin{equation*}\int_a^bf(x)dg(x)=e^{i\pi\alpha}\int_a^b\big(D_{a+}^{\alpha}f\big)(x)\big(D_{b-}^{1-\alpha}g_{b-}\big)(x)dx.
\end{equation*}
Moreover, we have the estimate
\begin{equation}\label{integr-estim}
\abs{\int_a^b f(x) dg(x)}\le C\norm{g}_{0,\alpha;[a,b]}\int_a^b \left(\frac{\abs{f(s)}}{(s-a)^{\alpha}} + \int_a^s \frac{\abs{f(s)-f(u)}}{(s-u)^{1+\alpha}}du \right)ds, 
\end{equation}
where 
$$
\norm{g}_{\alpha;[a,b]} = \sup_{a\le u<v\le b} \left(\frac{\abs{g(v)-g(u)}}{(v-u)^{1-\alpha}} + \int_u^v \frac{\abs{g(u)-g(z)}}{(z-u)^{2-\alpha}}dz\right).
$$
In what follows we fix some $\alpha\in(1-\gamma,\theta\wedge 1/2)$ and put $h(t,s)=(t-s)^{-1-\alpha}$.
Define $\norm{X}_{\infty,t} = \sup_{s\in[-r,t]}\abs{X(s)}$, $\norm{X}_{1,t} = \int_0^t \norm{X_{\cdot+t-s}-X_\cdot}_{\infty,s} h(t,s)ds$, $\norm{X}_{t} = \norm{X}_{\infty,t}+\norm{X}_{1,t}$. It is clear that both $\norm{X}_{\infty,t}$ and $\norm{X}_{1,t}$ are non-decreasing in $t$.

By a solution to equation \eqref{main-sdde}, we will understand a pathwise continuous $\mathbb{F}$-adapted process $X$ such that
$\norm{X}_T<\infty$ a.s., and \eqref{main-sdde} holds almost surely for all $t\in[0,T]$.

The following assumptions on the coefficients
of \eqref{main-sdde} will be assumed throughout the article:
\begin{enumerate}[H1.]
\item
Linear growth: for all $\psi\in\CC$, $t\in[0,T]$,
\begin{gather*}
|a(t,\psi)|+|b(t,\psi)|+|c(t,\psi)|\leq C(1+\norm{\psi}_{\CC}).
\end{gather*}
\item
For all $t\in[0,T]$, $\psi\in\CC$,  $c$ has a Fr\'echet derivative  $\partial_\psi c(t,\psi)\in L(\CC, \R^d)$, bounded uniformly in $t\in[0,T],\psi\in\CC$:
$$
\norm{\partial_{\psi}c(t,\psi)}_{L(\CC, \R^d)}\le C.
$$

\item The functions $a$, $b$ and $\partial_\psi c$ are locally Lipschitz continuous in $\psi$: for any $R>1$, $t\in[0,T]$, and all $\psi_1,\psi_2\in\CC$ with $\norm{\psi_1}_\CC\le R$,  $\norm{\psi_2}_\CC\le R$,
$$|a(t,\psi_1)-a(t,\psi_2)|+|b(t,\psi_1)-b(t,\psi_2)|+\norm{\partial_\psi c(t,\psi_1)-\partial_\psi c(t,\psi_2)}_{L(\CC, \R^d)}\leq C_R\norm{\psi_1-\psi_2}_{\CC}.$$
\item The functions $ c$ and $\partial_\psi c$ are H\"older continuous in $t$: for some  $\beta\in(1-\gamma,1)$ and for all $s,t\in[0,T]$, $\psi\in\CC$
$$|c(s,\psi)-c(t,\psi)|\le C|s-t|^\beta(1+\norm{\psi}_\CC),\quad \norm{\partial_\psi c(s,\psi)-\partial_\psi c(t,\psi)}_{L(\CC, \R^d)}\leq C|s-t|^\beta.$$
\end{enumerate}
The condition H4 allows, for instance, to consider an important particular case, namely, a linear equation.

\section{Auxiliary results}
First we establish some a priori estimates for the solution of \eqref{main-sdde}. 

\begin{lemma}\label{prop-apriormoment}
Let $X$ be a solution of \eqref{main-sdde}, and $p\ge 1$, $N\ge 1$. Let also  $A_{N,t} = \set{\norm{Z}_{\alpha;[0,t]}\le N}$ for $t\in[0,T]$.
 Then 
$$
\ex{\norm{X}_{T}^p\ind{A_{N,T}}}\le C_{N,p}.
$$
\end{lemma}
\begin{proof} 
Assume without loss of generality that $p>4/(1-2\alpha)$.

For $R>0$ define $B_{R,t} = \set{\norm{X}_{\infty,t} + \norm{X}_{1,t}\le R}$ and $\ind{t} = \ind{A_{N,t}\cap B_{R,t}}$.

Let $\omega\in A_{N,t}$. Write for $t\in[0,T]$
\begin{gather*}
\abs{X(t)}\le \abs{X(0)} + \abs{I^a(t)} + \abs{I^b(t)} + \abs{I^c(t)},
\end{gather*}
where $I^a(t) = \int_0^t a(s,X_s) ds$, $I^b(t)=\int_0^t b(s,X_s) dW(s)$, $I^c(t) = \int_0^t c(s,X_s) dZ(s)$. Estimate, using \eqref{integr-estim},
\begin{align*}
\abs{I^a(t)} &\le \int_0^t \abs{a(s,X_s)}ds \le C\int_0^t \left(1+\norm{X_s}_{\CC}\right)ds\le C\left(1+\int_0^t \norm{X}_{\infty,s}ds\right);\\
\abs{I^c(t)} &\le CN\int_0^t\left(\abs{c(s,X_s)}s^{-\alpha} + \int_0^s \abs{c(s,X_s) - c(u,X_u)} h(s,u)du\right)ds\\
&\le C N \int_0^t\left(\left(1+\norm{X_s}_{\CC}\right)s^{-\alpha} + \int_0^s \left(\abs{s-u}^{\beta}(1+\norm{X_s}_\CC) + \norm{X_s - X_u}_{\CC}\right) h(s,u)du\right)ds\\
&\le CN\left(1+ \int_0^t\left(\norm{X}_{\infty,s}s^{-\alpha} + \norm{X}_{1,s}\right)ds \right).
\end{align*}
Therefore, we have
\begin{equation*}
\abs{X(t)}\le CN\left(1+ \int_0^t\left(\norm{X}_{\infty,s} s^{-\alpha} + \norm{X}_{1,s}\right)ds \right) + \abs{I^b(t)},
\end{equation*}
whence 
\begin{equation}\label{normxit}
\norm{X}_{\infty,t}\le  CN\left(1+\int_0^t\left(\norm{X}_{\infty,s} s^{-\alpha} + \norm{X}_{1,s}\right)ds \right) + \norm{I^b}_{\infty;[0,t]} .
\end{equation}
Further, let $0\le s\le t$. Then for $u\le s-t$,
\begin{align*}
\abs{X(u+t-s)-X(u)} = \abs{\eta(u+t-s)-\eta(u)}\le H_\eta(t-s)^\theta,
\end{align*}
where $H_\eta = \sup_{-r\le x<y\le 0}\frac{\abs{\eta(y)-\eta(x)}}{(y-x)^\theta}$ is the $\theta$-H\"older seminorm of $\eta$. Similarly, for $u\in(s-t,0]$,
\begin{align*}
\abs{X(u+t-s)-X(u)} & \le \abs{X(u+t-s)-X(0)} +\abs{\eta(0)-\eta(u)}\\
& \le \abs{X(u+t-s)-X(0)} + H_\eta(t-s)^\theta.
\end{align*}
Consequently, we can write
$$
\norm{X}_{1,t} \le  H_\eta\int_{0}^{t}(t-s)^{\theta+\alpha-1} ds +  J^a(t) + J^b(t) + J^c(t)\le C + J^a(t) + J^b(t) + J^c(t),$$
where $J^b(t) = \int_0^t \sup_{u\in[s-t,s]}\abs{\int_{u\vee 0}^{u+t-s} b(v,X_v) dW(v)}h(t,s)ds$,
\begin{align*}
J^a(t) &= \int_0^t \sup_{u\in[s-t,s]}\abs{\int_{u\vee 0}^{u+t-s} a(v,X_v) dv }h(t,s)ds\\&\le C\int_0^t \max_{u\in[s-t,s]}\int_{u\vee 0}^{u+t-s} \big(1+\norm{X_{v}}_{\CC}\big)dv\,h(t,s)ds \\
&\le C\left(1+\int_{0}^{t}\int_{s}^{t} \norm{X}_{\infty,z} dz\,h(t,s)ds \right)\le C\left(1+\int_{0}^{t}  \norm{X}_{\infty,z} (t-z)^{-\alpha}dz \right);
\\ J^c(t)&= \int_0^t \sup_{u\in[s-t,s]}\abs{\int_{u\vee 0}^{u+t-s} c(v,X_v) dZ(v)}h(t,s)ds \le CN(J^c_1(t)+J^c_2(t)
\end{align*}
with
\begin{align*}
J^c_1(t) & = \int_0^t \max_{u\in[s-t,s]}\int_{u\vee 0}^{u+t-s}\abs{c(v,X_v)}(v-u\vee 0)^{-\alpha}dv\, h(t,s)ds \\&\le C\int_0^t \max_{u\in[-r,s]}\int_{u\vee 0}^{u+t-s}\big(1+ \norm{X_{v}}_{\CC}\big)(v-u\vee 0)^{-\alpha}dv\, h(t,s)ds\\
& \le C\left(1+\int_0^t \int_s^t \norm{X}_{\infty,z} (z-s)^{-\alpha}dz\,h(t,s)ds\right)\\
& \le C \left(1+\int_0^t \norm{X}_{\infty,z} (t-z)^{-2\alpha}dz\right);\\
J^c_2(t) & = \int_0^t \max_{u\in[-r,s]}\int_{u\vee 0}^{u+t-s}\int_{u\vee 0}^{v}
\abs{c(v,X_v)-c(z,X_z)}h(v,z)dz\, dv\, h(t,s)ds\\
&\le C\int_0^t \max_{u\in[-r,s]}\int_{u\vee 0}^{u+t-s}\int_{u\vee 0}^{v}
\left(\abs{v-z}^\beta +
\norm{X_v-X_z}_\CC\right)h(v,z) dz\,dv\, h(t,s)ds
\\ &\le C\int_0^t \max_{u\in[-r,s]}\int_{u\vee 0}^{u+t-s}\left(\abs{v-u\vee 0}^{\beta-\alpha} + \norm{X}_{1,v}\right)dv\, h(t,s)ds\\
& \le C \int_0^t \left((t-s)^{\beta-2\alpha} + \int_s^t \norm{X}_{1,v}dv\,h(t,s) \right)ds \le C \left(1+ \int_0^t \norm{X}_{1,v}(t-v)^{-\alpha}dv \right).
\end{align*}
To estimate $J^c_1$, we have used the computation
\begin{align*}
&\int_0^z (z-s)^{-\alpha} (t-s)^{-1-\alpha} ds  = \Big|s = z - (t-z)v \Big| = (t-z)^{-2\alpha} \int_0^{\frac{z}{t-z}} v^{-\alpha} (1+v)^{-1-\alpha} dv \\&\qquad\le
(t-z)^{-2\alpha} \int_0^{\infty} v^{-\alpha} (1+v)^{-1-\alpha} dv = \mathrm{B}(1-\alpha,2\alpha) (t-z)^{-2\alpha}.
\end{align*}
Summing the estimates for $\norm{X}_{1,t}$, we get
\begin{equation}\label{normx1t}
\norm{X}_{1,t} \le CN\left(1+ \int_0^t\left(\norm{X}_{\infty,s}(t-s)^{-2\alpha} + \norm{X}_{1,s}(t-s)^{-\alpha}\right)ds \right) + J^b(t).
\end{equation}
Combining this with \eqref{normxit},  we obtain 
$$
\norm{X}_t\le CN\int_0^t \norm{X}_sg(t,s)ds + \norm{I^b}_{\infty;[0,t]} + J^b(t)
$$
for $\omega \in A_{N,t}$, where $g(t,s) = s^{-\alpha}+(t-s)^{-2\alpha}$.

Using the H\"older inequality, we can estimate
\begin{align*}
\norm{X}^p_t\le  C_{p}N^p\int_0^t \norm{X}^p_s g(t,s)ds
\left(\int_0^tg(t,s)ds\right)^{p/q}  + C_p\left(\norm{I^b}_{\infty;[0,t]}^p + \left(J^b(t)\right)^p\right),
\end{align*}
whence
\begin{equation}\label{normxt}
\ex{\norm{X}^p_t\ind{t}}\le C_{N,p}\left(\int_0^t \ex{\norm{X}^p_s\ind{s}}g(t,s)ds
  + \ex{\norm{I^b}_{\infty;[0,t]}^p\ind{t}} + \ex{\left(J^b(t)\right)^p\ind{t}}\right).
\end{equation}
We now proceed to the estimation of the last two expressions.
It is obvious that for any $0\le u\le s\le t$, 
$$
\abs{\int_{u}^{s} b(v,X_v)dW(v)}\ind{t} \le 
 \abs{\int_{u}^{s} b(v,X_v)\ind{v}dW(v)}.
$$
Therefore, by the Burkholder inequality,
\begin{align*}
& \ex{\norm{I^b}_{\infty;[0,t]}^p\ind{t}} = \ex{\sup_{s\in[0,t]}\abs{\int_0^s b(v,X_v)dW(v) }^p\ind{t}}\le \ex{\sup_{s\in[0,t]}\abs{\int_0^s b(v,X_v)\ind{v}dW(v) }^p}\\
&\qquad \le C_p\ex{\left(\int_0^t \abs{b(s,X_s)}^2\ind{s}ds\right)^{p/2}}\le C_p\int_0^t\ex{\left(1+ \norm{X_s}_\CC\ind{s}\right)^p} ds\\&\qquad \le C_p\int_0^t\left(1+\ex{ \norm{X_s}_{\CC}^p\ind{s}}\right)ds\le C_p \left(1+ \int_{0}^{t}\ex{ \norm{X}_{\infty,s}^p\ind{s}}ds\right).
\end{align*}
Further, we have
\begin{equation}\label{Jbt}
\ex{\left(J^b(t)\right)^p\ind{t}} \le C_p \ex{\left(\int_0^t\sup_{u\in[s-t,s]}\abs{ \int_{u\vee 0}^{u+t-s} b(v,X_v)\ind{v} dW(v)} h(t,s)ds \right)^p}
\end{equation}

It follows from the Garsia--Rodemich--Rumsey inequality that
for any $r,z\in[0,t]$
$$\abs{\int_{r}^{z} b(v,X_v)\ind{v} dW(v)}\le C_p \xi(t)\abs{r-z}^{1/2-2/p},$$ where
$$
\xi(t) = \left(\int_0^t \int_0^y \frac{\abs{\int_x^y b(v,X_v)\ind{v} dW(v)}^{p}} {\abs{x-y}^{p/2}}dx\,dy\right)^{1/p}.
$$
We can estimate
\begin{align*}
&\ex{\xi(t)^p}= \int_0^t \int_0^y \frac{\ex{\abs{\int_x^y b(v,X_v)\ind{v} dW(v)}^{p}}} {\abs{x-y}^{p/2}}dx\,dy\\
&\qquad\le C_p \int_0^t \int_0^y \frac{\ex{\left(\int_x^y (1+\norm{X_v}_{\CC}^2)\ind{v}dv\right)^{p/2}}} {(y-x)^{p/2}} dx\,dy\\
&\qquad\le C_p \int_0^t \int_0^y \frac{(y-x)^{p/2-1}\ex{\int_x^y \big(1+\norm{X}_{\infty,v}^p\ind{v}\big)dv}} {(y-x)^{p/2}} dx\,dy 
\\
&\qquad\le C_p \left(1+\int_0^t \int_0^y\ex{\norm{X}_{\infty,v}^p\ind{v}}\int_0^v   {(y-x)^{-1}}dx\, dv\,dy\right) \\
&\qquad= C_p \left(1+ \int_0^t \ex{\norm{X}_{\infty,v}^p\ind{v}} \int_v^t  \log \frac{y}{y-v}dy\, dv \right)\\
&\qquad\le C_p \left(1+ \int_0^t \ex{\norm{X}_{\infty,v}^p\ind{v}}dv\right).
\end{align*}
Therefore, taking into account that $p>4/(1-2\alpha)$, i.e. $2/p+1/\alpha-1/2<0$, we get from \eqref{Jbt}
\begin{align*}
&\ex{J^b(t)^p\ind{t}}\le C_p \ex{\xi(t)^p} \left(\int_0^t (t-s)^{-2/p-1/2-\alpha}ds\right)^p\le C_p \left(1+ \int_0^t \ex{\norm{X}_{\infty,v}^p\ind{v}}dv\right).
\end{align*}
Plugging the estimates of $I^b$ and $J^b$ into \eqref{normxt}, we get 
$$\ex{\norm{X}^p_t\ind{t}}\le C_{N,p}\left(1+\int_0^t \ex{\norm{X}^p_s\ind{s}}g(t,s)ds\right).
$$
Since $g(t,s)\le (T^{\alpha}+1)t^{2\alpha}s^{-2\alpha}(t-s)^{-2\alpha}$, 
we can apply the generalized Gronwall lemma \cite[Lemma 7.6]{NR} and obtain $\ex{\norm{X}^p_T\ind{T}}\le C_{N,p}$. By letting $R\to\infty$ and using the Fatou lemma, we arrive at the required statement.
\end{proof}

The following lemma establishes estimates for the distance between solutions of mixed stochastic delay differential equations with different drivers. To formulate it, assume that $\overline{Z}$ is another $\gamma$-H\"older $\mathbb{F}$-adapted process, and consider the equation
\begin{equation}\label{sdde-oX}
\oX(t) = X(0) + \int_{0}^{t} a(s,\oX_s)ds +  \int_0^t b(s,\oX_s)dW(s) + \int_0^t c(s,\oX_s)d\oZ(s)
\end{equation}
with the same initial condition $\oX(s) = \eta(s)$, $s\in[-r,0]$.  
\begin{lemma}\label{prop-quasicontract}
Let $X$ and $\oX$ be  solutions of \eqref{main-sdde} and \eqref{sdde-oX} respectively, $p\ge 4/(1-2\alpha)$, $N\ge 1$, $R\ge 1$. Assume also  that $\norm{Z}_{\alpha;[0,T]}\le N$ and $\norm{\oZ}_{\alpha;[0,T]}\le N$. Then
$$
\ex{\norm{X-\oX}_{\infty,T}^p\ind{B_{R,T}}}\le C_{N,R,p}\ex{\norm{Z-\oZ}^p_{\alpha;[0,T]}},
$$
where
$B_{R,t} = \set{\norm{X}_{t}\le R, \norm{\oX}_{t}\le R}$ for $t\in[0,T]$.
\end{lemma}
\begin{proof} 
The proof will be similar to that of Lemma~\ref{prop-apriormoment}, so we will omit some details. Put $\Delta(t) = \norm{X-\oX}_t$, $\Delta_d(t) = d(s,X_s) - d(s,\oX_s)$ for $d\in \set{a,b,c}$, and $\Delta_Z(t)= Z(t)-\oZ(t)$. By assumption H3, $\Delta_d(t)\le C_R\norm{X_t-\oX_t}_\CC \le C_R\Delta(t)$.

Let $\omega\in B_{R,t}$.  Write for $t\in[0,T]$
\begin{gather*}
\abs{X(t)-\oX(t)}\le  \abs{I^a(t)} + \abs{I^b(t)} + \abs{I^c(t)} + \abs{I^Z(t)},
\end{gather*}
where $I^a(t) = \int_0^t \Delta_a(s) ds$, $I^b(t)=\int_0^t \Delta_b(s) dW(s)$, $I^c(t) = \int_0^t \Delta_c(s) dZ(s)$, $I^Z(t) = \int_0^t c(s,\oX_s)d\Delta_Z(t)$. We estimate the terms one by one, starting with $I^a$:
\begin{align*}
\abs{I^a(t)} &\le \int_0^t \abs{\Delta_a(s)}ds \le C_R\int_0^t \Delta(s)ds
\end{align*}
Similarly to $I^c(t)$ in the proof of Lemma~\ref{prop-apriormoment},
\begin{align*}
\abs{I^Z(t)}& \le C\norm{\Delta_Z}_{\alpha;[0,t]}\int_0^t 
\left(\norm{\oX}_{\infty,s}s^{-\alpha} +\norm{\oX}_{1,s}\right)ds\le 
CR\norm{\Delta_Z}_{\alpha;[0,t]}.
\end{align*}
Further,
\begin{align*}
\abs{I^c(t)} &\le CN\int_0^t\left(\abs{\Delta_c(s)}s^{-\alpha} + \int_0^s \abs{\Delta_c(s) - \Delta_c(u)} h(s,u)du\right)ds\\
&\le C_R N \int_0^t\left(\Delta(s)s^{-\alpha} + \int_0^s \abs{\Delta_c(s) - \Delta_c(u)} h(s,u)du\right)ds.
\end{align*}
Similarly to \cite[Lemma 7.1]{NR}, it can be shown that assumptions H3 and H4 imply that for any $s,u\in[0,T]$ and $\psi_1,\dots,\psi_4\in\CC$ with $\norm{\psi_i}\le R$, $i=1,\dots,4$,
\begin{equation}\label{cspsi}
\begin{gathered}
\abs{c(s,\psi_1)-c(u,\psi_2) - c(s,\psi_3) + c(u,\psi_4)}\le C_R\Big(\norm{\psi_1-\psi_2-\psi_3+\psi_4}_\CC \\ + \norm{\psi_1-\psi_3}_\CC\big(\abs{s-u}^\beta+ \norm{\psi_1-\psi_2}_\CC + \norm{\psi_3-\psi_4}_\CC\big)\Big).
\end{gathered}
\end{equation}Therefore, we can estimate $\abs{I^c(t)}\le C_RN\sum_{k=1}^d I^c_k(t)$, where
\begin{align*}
I^c_1(t) &= \int_0^t \Delta(s)s^{-\alpha}ds;\\
I^c_2(t) &= \int_0^t \int_0^s \norm{X_s-\oX_s - \oX_u+ \oX_s}_\CC h(s,u)du\,ds \le \int_0^t \norm{X-\oX}_{1,s}ds\le \int_0^t \Delta(s)ds;
\\
I^c_3(t) &= \int_0^t \int_0^s \norm{X_s-\oX_s}_\CC (s-u)^{\beta-\alpha-1}du\,ds\le C\int_0^t \norm{X-\oX}_{\infty,s}ds\le C\int_0^t \Delta(s)ds;
\\
I^c_4(t) &= \int_0^t \int_0^s \norm{X_s-\oX_s}_\CC\left(\norm{X_s-X_u}_\CC 
+ \norm{\oX_s-\oX_u}_\CC\right)h(s,u)du\\
&\le \int_0^t \norm{X_s-\oX_s}_{\infty,s}\left(\norm{X}_{1,s} + \norm{\oX}_{1,s}\right)\le 2R\int_0^t \Delta(s)ds.
\end{align*}
Therefore, we have
\begin{equation}\label{normxoxit}
\norm{X-\oX}_{\infty,t}\le C_{N,R}\left(\norm{\Delta_Z}_{\alpha;[0,t]}+ \int_0^t\Delta(s) s^{-\alpha}ds\right)  + \norm{I^b}_{\infty;[0,t]}.
\end{equation}

Further, let $0\le s\le t$. Then for $u\le s-t$,
\begin{align*}
\abs{X(u+t-s) -\oX(u+t-s) - X(u) + \oX(u)} = 0;
\end{align*}
for $u\in(s-t,0]$
\begin{align*}
\abs{X(u+t-s) -\oX(u+t-s) - X(u) + \oX(u)} = \abs{X(u+t-s)-\oX(u+t-s)}.
\end{align*}
Consequently, we can write
$$
\norm{X-\oX}_{1,t} \le  J^a(t) + J^b(t) + J^c(t) + J^Z(t),$$
where  
$J^a(t) = \int_0^t \sup_{u\in[s-t,s]}\abs{\int_{u\vee 0}^{u+t-s} \Delta_a(v) dv }$, $J^b(t) = \int_0^t \sup_{u\in[s-t,s]}\abs{\int_{u\vee 0}^{u+t-s} \Delta_b(v) dW(v)}ds$, $J^c(t) = \int_0^t \sup_{u\in[s-t,s]}\abs{\int_{u\vee 0}^{u+t-s} \Delta_c(v) dZ(v) }$, $J^Z(t) = \int_0^t \sup_{u\in[s-t,s]}\abs{\int_{u\vee 0}^{u+t-s} c(\oX_v,v) d\Delta_Z(v) }$. Estimate
\begin{align*}
J^a(t) &\le C_R\int_0^t \max_{u\in[s-t,s]}\int_{u\vee 0}^{u+t-s} \Delta(v)dv\,h(t,s)ds \le C\int_{0}^{t}  \Delta(z) (t-z)^{-\alpha}dz.
\end{align*}
Similarly to $J^c(t)$  in the proof of Lemma~\ref{prop-apriormoment},
\begin{align*}
J^Z(t)& = C\norm{\Delta_Z}_{\alpha;[0,t]}\left(1+\int_0^t 
\big(\norm{\oX}_{\infty,s}(t-s)^{-2\alpha} +\norm{\oX}_{1,s}(t-s)^{-\alpha}\big)ds\right)\le C_R \norm{\Delta_Z}_{\alpha;[0,t]}.
\end{align*}
Further, using \eqref{cspsi}, we can estimate, analogously to $I^c(t)$ above,
$J^c(t) \le C_{N,R}\sum_{k=1}^4 J^c_k(t)$, where
\begin{align*}
J^c_1(t) & = \int_0^t \max_{u\in[s-t,s]}\int_{u\vee 0}^{u+t-s}\abs{\Delta_c(v)}(v-u\vee 0)^{-\alpha}dv\, h(t,s)ds \\
& \le C_R\int_0^t \int_s^t \Delta(z) (z-s)^{-\alpha}dz\,h(t,s)ds \le C \int_0^t \Delta(z) (t-z)^{-2\alpha}dz;\\
J^c_2(t) & = \int_0^t \max_{u\in[-r,s]}\int_{u\vee 0}^{u+t-s}\int_{u\vee 0}^{v}
\norm{X -\oX - X +\oX}_{\infty,z}h(v,z)dz\, dv\, h(t,s)ds\\
&\le C\int_0^t \max_{u\in[-r,s]}\int_{u\vee 0}^{u+t-s} \norm{X-\oX}_{1,v} dv\, h(t,s)ds\\
& \le C \int_0^t \int_s^t \norm{X-\oX}_{1,v}dv\, h(t,s) ds \le C  \int_0^t \Delta(v)(t-v)^{-\alpha}dv;\\
J^c_3(t) & = \int_0^t \max_{u\in[-r,s]}\int_{u\vee 0}^{u+t-s}\int_{u\vee 0}^{v}
\norm{X -\oX}_{\infty,v}(v-z)^{\beta-\alpha-1}dz\, dv\, h(t,s)ds\\
&\le C\int_0^t \max_{u\in[-r,s]}\int_{u\vee 0}^{u+t-s} \norm{X-\oX}_{\infty,v} dv\, h(t,s)ds\\
&\le C\int_0^t \int_s^t \norm{X-\oX}_{\infty,v} dv\, h(t,s)ds \le C  \int_0^t \Delta(v)(t-v)^{-\alpha}dv;\\
J^c_4(t) & = \int_0^t \max_{u\in[-r,s]}\int_{u\vee 0}^{u+t-s}\int_{u\vee 0}^{v}
\norm{X -\oX}_{\infty,v}\left(\norm{X_v-X_z}_\CC+ \norm{\oX_v-\oX_z}_\CC\right)h(v,z)dz\, dv\, h(t,s)ds\\
&\le C\int_0^t \max_{u\in[-r,s]}\int_{u\vee 0}^{u+t-s} \norm{X-\oX}_{\infty,v} \left(\norm{X}_{1,v}+\norm{\oX}_{1,v}\right) dv\, h(t,s)ds\\&
\le CR\int_0^t \int_s^t \norm{X-\oX}_{\infty,v} dv\, h(t,s)ds\le C R \int_0^t \Delta(v)(t-v)^{-\alpha}dv.
\end{align*}
Summing the estimates for $\norm{X-\oX}_{1,t}$, we get
\begin{equation*}
\norm{X-\oX}_{1,t} \le C_{N,R}\left(\norm{\Delta_Z}_{\alpha;[0,t]}+ \int_0^t \Delta(s)(t-s)^{-2\alpha}ds \right) + J^b(t).
\end{equation*}
Combining this with the estimate\eqref{normxoxit},  we obtain 
$$
\norm{X-\oX}_t\le C_{N,R}\left(\norm{\Delta_Z}_{\alpha;[0,t]}+\int_0^t \norm{X}_s g(t,s)ds\right) + \norm{I^b}_{\infty;[0,t]} + J^b(t)
$$
for $\omega \in B_{R,t}$, where $g(t,s) = s^{-\alpha}+(t-s)^{-2\alpha}$.
The rest of the proof goes exactly as in the Lemma~\ref{prop-apriormoment}. Namely, denoting $\ind{t} = \ind{B_{R,t}}$, we obtain
$$\ex{\norm{X}^p_t\ind{t}}\le C_{N,p}\left(\ex{\norm{\Delta_Z(t)}_{\alpha;[0,t]}^p}+ \int_0^t \ex{\norm{X}^p_s\ind{s}}g(t,s)ds\right),
$$
which implies the required statement with the help of the generalized Gronwall lemma.
\end{proof}
\section{Existence and uniqueness of solution}\label{sec:exuniq}
Now we have everything to establish the unique solvability of   \eqref{main-sdde}.
\begin{theorem}\label{thm:exuni}
Equation \eqref{main-sdde} has a unique solution. 
\end{theorem}
\begin{proof}
For convenience, the proof will be divided into several logical steps.

\textit{Step 1. Approximations by usual stochastic delay differential equations}

Fix some $N\ge 1$ and define $\tau_N = \inf\set{t>0\colon \norm{Z}_{\alpha;[0,t]}\ge N}$, $Z^N(t) = Z(t\wedge \tau_N)$, $t\ge 0$. For each integer $n\ge 1$ define a smooth 
approximation of $Z^N$ by 
$$
Z^{N,n}(t) = n\int_{(t-1/n)\vee 0}^{t} Z^N(s)ds
$$
and consider the equation 
\begin{equation*}
X^{N,n}(s) = X(0) + \int_{0}^{t} a(s,X^{N,n}_s)dt +  \int_0^t b(s,X^{N,n}_s)dW(s) + \int_0^t c(s,X^{N,n}_s)dZ^{N,n}(s)
\end{equation*}
with the same initial condition $X^{N,n}(s) = \eta(s)$, $s\in[-r,0]$. Since $Z^{N,n}$ is absolutely continuous, this is a usual stochastic delay differential equation (or, in the terminology of \cite{mohammed}, stochastic functional differential equation)
\begin{equation}\label{ito-sdde}
X^{N,n}(s) = X(0) + \int_{0}^{t} d^{N,n}(s,X^{N,n}_s) dt +  \int_0^t b(s,X^{N,n}_s)dW(s)
\end{equation}
with a random drift $d^{N,n}(s,\psi) = a(s,\psi) + c(s,\psi)\frac{d}{ds}Z^{N,n}(s)$. Clearly, $\abs{\frac{d}{ds}Z^{N,n}(s)}\le nN$. Therefore, the coefficients of \eqref{ito-sdde} satisfy the linear growth condition: for all $s\in[0,T]$, $\psi\in\CC$,
\begin{equation}\label{lingrowth}
\abs{d^{N,n}(s,\psi)}+\abs{b(s,\psi)}\le C_{N,n}\left(1+\norm{\psi}_\CC\right),
\end{equation}
and the local Lipschitz condition: for any $R>0$ and all $s\in[0,T]$, $\psi_1,\psi_2\in\CC$ with $\norm{\psi_1}_\CC\le R, \norm{\psi_2}_\CC\le R$,
\begin{equation}\label{lipschitz}
\abs{d^{N,n}(s,\psi_1)-d^{N,n}(s,\psi_2)}+\abs{b(s,\psi_1)-b(s,\psi_2)}\le C_{N,n,R}\norm{\psi_1-\psi_2}_\CC.
\end{equation}
In \cite[Theorem I.2]{mohammed} and in \cite[Chapter 5, Theorem 2.5]{mao}, the unique solvability of \eqref{ito-sdde} was formulated for non-random coefficients satisfying conditions \eqref{lingrowth} and \eqref{lipschitz}. However, the arguments given there are easily seen to extend to adapted coefficients satisfying \eqref{lingrowth} and \eqref{lipschitz} with  a non-random constant, which is the case here.  Thus, \eqref{ito-sdde} has a unique solution.
  
\textit{Step 2. Convergence of approximations}

First we show that, for a fixed $N\ge 1$, the sequence $\set{X^{N,n},n\ge 1}$ is fundamental in probability in the norm $\norm{\cdot}_T$. Indeed, it is easy to show (see e.g.\ \cite[Lemma 2.1]{mbfbm-limit}) that $\norm{Z^{N,n}-Z^N}_{\alpha;[0,T]}\to 0$, $n\to\infty$, a.s. Then, in view of the boundedness, $\ex{\norm{Z^{N,n}-Z^N}_{\alpha;[0,T]}^p}\to 0$ for any $p\ge 1$. Therefore,   Lemma~\ref{prop-quasicontract} and the Markov inequality imply that
\begin{equation}\label{fundprob}
\mathsf{P}\left(\norm{X^{N,n}-X^{N,m}}_T>\eps,\norm{X^{N,n}}_T\le R,\norm{X^{N,m}}_T\le R\right)\to 0,\quad n,m\to\infty,
\end{equation}
for any $\eps>0$, $R\ge 1$. Hence,
$$
\limsup_{n,m\to\infty}\mathsf{P}\left(\norm{X^{N,n}-X^{N,m}}_T>\eps\right)\le 2\sup_{n\ge 1}\mathsf{P}\left(\norm{X^{N,n}}_T> R\right)
$$
for any $\eps>0$, $R\ge 1$. The convergence $\ex{\norm{Z^{N,n}-Z^N}_{\alpha;[0,T]}^p}\to 0$, $n\to\infty$ implies that $\sup_{n\ge 1}\ex{\norm{Z^{N,n}}_{\alpha;[0,T]}^p}<\infty$. Then, due to Lemma~\ref{prop-apriormoment} and the Markov inequality,
$$
\sup_{n\ge 1}\mathsf{P}\left(\norm{X^{N,n}}_T> R\right)\to 0,\quad n\to\infty,
$$
whence, letting $R\to\infty$ in \eqref{fundprob}, we deduce $\mathsf{P}\left(\norm{X^{N,n}-X^{N,m}}_T>\eps\right)\to 0$, $n,m\to\infty$, as required. Therefore, there exists some random process $X^N$ such that $\norm{X^{N,n}-X^N}_{T}\to 0$, $n\to\infty$, in probability. There is an almost surely convergent subsequence, and without loss of generality we can assume that $\norm{X^{N,n}- X^N}_T\to 0$, $n\to\infty$, a.s.

\textit{Step 3. The limit provides a solution.}

In order to prove that $X^N$ solves equation \eqref{main-sdde} with $Z$ replaced by $Z^N$, we need to show that the integrals in \eqref{ito-sdde} converge to the correspondent integrals for $X^N$. 
Since the convergence $\norm{X^{N,n}- X^N}_T\to 0$, $n\to\infty$, implies the uniform convergence on $[0,T]$, we easily obtain
$$
\int_{0}^{t} a(s,X^{N,n}_s)ds \to \int_{0}^{t} a(s,X^N_s)ds,\quad n\to\infty,\quad \text{a.s.}
$$
Similarly to $I^c(t)$ and $I^Z(t)$ in the proof of Lemma~\ref{prop-quasicontract}, we have
\begin{align*}
&\abs{\int_0^t c(s,X^N_s)dZ^N(s)-\int_0^t c(s,X^{N,n}_s)dZ^{N,n}(s)}\\
&\qquad \le C_{N}\left(\norm{X^N}_{t}+ \norm{X^{N,n}}_{t}\right)\left(\norm{Z^N-Z^{N,n}}_{\alpha;[0,t]} + \int_0^t \norm{X^N - X^{N,n}}_{t}ds\right)\to 0
\end{align*}
as $n\to\infty$ a.s. Finally, denoting $\ind{t}=\ind{\norm{X^N}_t< R,\norm{X^{N,n}}_t< R}$, we have
\begin{align*}
&\ex{\left(\int_0^t b(s,X^N_s)dW(s)-\int_0^t b(s,X^{N,n}_s)dW(s)\right)^2\ind{t}}\\
& \qquad \le \int_{0}^{t}\ex{\left(b(s,X^{N}_s)-b(s,X^{N,n}_s)\right)^2\ind{s}}ds
\\
& \qquad \le 
\int_{0}^{t}\ex{\norm{X^{N}-X^{N,n}}^2_s\ind{s}}ds\to 0,\quad n\to\infty.
\end{align*}
So we have that $$\left(\int_0^t b(s,X^N_s)dW(s)-\int_0^t b(s,X^{N,n}_s)dW(s)\right)\ind{t}\to 0,\quad n\to\infty$$ in probability. Thanks to the convergence $\norm{X^{N,n}-X^N}_{T}\to 0$, $n\to\infty$, the event $\set{\norm{X^N}_t< R}$ implies $\set{\norm{X^{N,n}}_t< R}$ for $n$ large enough, therefore we have the convergence of the integrals in probability on $\set{\norm{X^N}_t< R}$ and arbitrary $R\ge 1$, therefore on $\Omega$. Thus, we have that $X^N$ is a solution to 
$$
X^{N}(s) = X(0) + \int_{0}^{t} a(s,X^{N}_s) dt +  \int_0^t b(s,X^{N}_s)dW(s) + \int_0^t c(s,X^N_s)dZ^N(s)
$$
with $X^N(s) = \eta(s)$, $s\in[-r,0]$. 

From Lemma~\ref{prop-quasicontract}, it is obvious that the processes
$X^N$ and $X^M$ with $M\ge N$ coincide a.s.\ on the set $A_{N,R}=\set{\norm{Z}_{\alpha;[0,T]}\le N}$. Therefore, there exists a process $X$ such that for each $N\ge 1$, $X^N = X$ a.s.\ on $A_{N,T}$. Consequently, $X$ solves \eqref{main-sdde} on each of the sets $A_{N,T}$, $N\ge 1$, hence, almost surely. 

Finally, the uniqueness follows from Lemma~\ref{prop-quasicontract}: each solution to \eqref{main-sdde} must coincide with $X$ on each of the sets $A_{N,T}$, hence, almost surely.
\end{proof}

\section{Integrability and convergence of solutions}

Now we investigate the question when the moments of $X$ are finite. Naturally, we need to require certain integrability of the driver $Z$. It is quite involved to prove the integrability under assumptions H1--H4 (for equations without delay, the corresponding result is proved in \cite{mbfbm-integr}). So we prove the integrability under an additional assumption that $b$ is bounded.
\begin{theorem}\label{thm:moments}
Assume, that, in addition to H1--H4, $\abs{b(t,\psi)}\le C$ for any $t\in[0,T]$, $\psi\in \CC$,
and $\ex{\exp\set{c\norm{Z}_{\alpha;[0,T]}^{1/(1-\alpha)}}}<\infty$
for all $c>0$. Then the solution of \eqref{main-sdde} satisfies $\ex{\norm{X}_T^p}<\infty$
for all $p\ge 1$, in particular, all moments of the solution are finite.
\end{theorem}
\begin{proof}The proof follows the scheme of \cite[Lemma 4.1]{mbfbm-jumps}. We will use the notation of Lemma~\ref{prop-apriormoment}. 

Define for $\lambda>0$, $t\in[0,T]$, $a\in\set{1,\infty}$ \ $\norm{X}_{\lambda;a} = \sup_{s\in[0,T]}e^{-\lambda s}\norm{X}_{a,s}$. 
Denote also $\zeta = \norm{I^b}_{\infty;[0,T]} + J^b(T)$.
Then from 
\eqref{normx1t} we get for $\omega\in A_{N,t}$
\begin{align*}
\norm{X}_{\lambda;\infty}&\le CN\left(1+ \sup_{s\le T}e^{-\lambda s}\int_0^s\left(\norm{X}_{\infty,u}u^{-\alpha} + \norm{X}_{1,u}\right)du \right) + \zeta\\
&\le CN\left(1+ \sup_{s\le T}\int_0^se^{\lambda(u-s)}\left(e^{-\lambda u}\norm{X}_{\infty,u}u^{-\alpha} + e^{-\lambda u}\norm{X}_{1,u}\right)du \right) + \zeta\\
&\le CN\left(1+ \sup_{s\le T}\int_0^s e^{\lambda(u-s)}\left(u^{-\alpha} \norm{X}_{\lambda;\infty}  + \norm{X}_{\lambda;1}\right)du \right) + \zeta\\
&\le
CN\left(1+ \lambda^{\alpha-1}\norm{X}_{\lambda;\infty}  + \lambda^{-1}\norm{X}_{\lambda;1}\right) + \zeta,
\end{align*}
where we have used the estimate
\begin{align*}
&\sup_{s\le T}\int_0^s e^{\lambda(u-s)}u^{-\alpha} du =\sup_{s\le T}\lambda^{-1}\int_0^{\lambda s} e^{-z}( s - z/\lambda)^{-\alpha} dz\\&\quad = \sup_{s\le T}\lambda^{\alpha-1}\int_0^{\lambda s} e^{-z}(\lambda s - z)^{-\alpha} dz\le  \lambda^{\alpha-1}\sup_{a>0}\int_0^{a} e^{-z}(a - z)^{-\alpha} dz =C\lambda^{\alpha-1}.
\end{align*}
Similarly, from \eqref{normxit},
\begin{align*}
&\norm{X}_{\lambda;1}\le N\left(1+ \sup_{s\le T}e^{-\lambda s}\int_0^s\left(\norm{X}_{\infty,u}(s-u)^{-2\alpha} + \norm{X}_{1,u}(s-u)^{-\alpha}\right)du \right) + \zeta\\
&\le N\left(1+ \sup_{s\le T}\int_0^s e^{\lambda(u-s)}\left(e^{-\lambda u}\norm{X}_{\infty,u}(s-u)^{-2\alpha} + e^{-\lambda u}\norm{X}_{1,u}(s-u)^{-\alpha}\right)du \right) + \zeta\\
&\le N\left(1+ \sup_{s\le T}\int_0^s e^{\lambda(u-s)}\left( \norm{X}_{\lambda;\infty}(s-u)^{-2\alpha}  + \norm{X}_{\lambda;1}(s-u)^{-\alpha}\right)du \right) + \zeta\\
&\le
CN\left(1+ \lambda^{2\alpha-1}\norm{X}_{\lambda;\infty}  + \lambda^{\alpha-1}\norm{X}_{\lambda;1}\right) + \zeta.
\end{align*}
Therefore, we have arrived at the system of inequalities
\begin{align*}
\norm{X}_{\lambda;\infty}&\le KN\left(1+ \lambda^{\alpha-1}\norm{X}_{\lambda;\infty}  + \lambda^{-1}\norm{X}_{\lambda;1}\right) + \zeta,\\
\norm{X}_{\lambda;t}&\le KN\left(1+ \lambda^{2\alpha-1}\norm{X}_{\lambda;\infty}  + \lambda^{\alpha-1}\norm{X}_{\lambda;1}\right) + \zeta.
\end{align*}
Setting $\lambda = 4KN^{1/(1-\alpha)}$, 
it is easy to deduce from this system that 
$$
\norm{X}_{\lambda;\infty} + \norm{X}_{\lambda;1}\le C N^{1/(1-\alpha)}(1+\zeta),
$$
whence 
$$
\norm{X}_T\le e^{\lambda T}\left(\norm{X}_{\lambda;\infty} + \norm{X}_{\lambda;1}\right)\le C\exp\set{C N^{1/(1-\alpha)}}(1+\zeta)
$$
for $\omega\in A_{N,t}$. 
Thus, in order to prove the required result, it remains to show that all moments of $\zeta$ are finite. The argument is similar to the estimation of $I^b$ and $J^b$ in Lemma~\ref{prop-apriormoment}, so we omit some details.

Take arbitrary $p>4/(1-2\alpha)$. By the Burkholder inequality,
\begin{equation}\label{ibt}
\begin{aligned}
 \ex{\norm{I^b}_{\infty;[0,T]}^p} \le C_p\ex{\left(\int_0^T \abs{b(s,X_s)}^2 ds\right)^{p/2}}<\infty.
\end{aligned}
\end{equation}
Further, we have
\begin{equation} 
\ex{J^b(T)^p} \le C_p \ex{\left(\int_0^T\sup_{u\in[s-T,s]}\abs{ \int_{u\vee 0}^{u+T-s} b(v,X_v) dW(v)} h(T,s)ds \right)^p}.
\end{equation}
By the Garsia--Rodemich--Rumsey inequality,
for any $r,z\in[0,T]$
$$\abs{\int_{r}^{z} b(v,X_v)\ind{v} dW(v)}\le C_p \xi(T)\abs{r-z}^{1/2-2/p},$$ where
$$
\xi(T) = \left(\int_0^t \int_0^y \frac{\abs{\int_x^y b(v,X_v) dW(v)}^{p}} {\abs{x-y}^{p/2}}dx\,dy\right)^{1/p}.
$$
From the estimate
\begin{align*}
&\ex{\xi(T)^p}= \int_0^T \int_0^y \frac{\ex{\abs{\int_x^y b(v,X_v) dW(v)}^{p}}} {\abs{x-y}^{p/2}}dx\,dy\\
&\qquad\le C_p \int_0^T \int_0^y \frac{\ex{\left(\int_x^y \abs{b(v,X_v)}^2 dv\right)^{p/2}}} {(y-x)^{p/2}} dx\,dy\le C_p \int_0^T \int_0^y 1 dx\,dy<\infty
\end{align*}
we obtain, as in Lemma~\ref{prop-apriormoment}, $\ex{J^b(T)^p}<\infty$. Taking into account \eqref{ibt}, we get that $\ex{\zeta^p}<\infty$, thus finishing the proof.
\end{proof}
\begin{remark}
The assumption on $Z$ from Theorem~\ref{thm:moments} is fulfilled e.g.\ for a fractional Brownian motion with Hurst parameter $H>1/2$ (with any $\alpha>1-H$), see e.g.\ \cite[Theorem 4]{mbfbm-jumps}.
\end{remark}

Finally, we state a result on stability of solutions to \eqref{main-sdde} with respect to the driver $Z$. Its proof virtually repeats Step 3 of the proof of Theorem~\ref{thm:exuni} and therefore is omitted.
Let for $n\ge 1$ \ $Z^n=\set{Z^n(t),n\ge 1}$ be an $\mathbb F$-adapted $\gamma$-H\"older continuous process, and $X^n$ be a solution to 
\begin{equation}\label{sdde-n}
X^n(s) = X(0) + \int_{0}^{t} a(s,X^n_s)dt +  \int_0^t b(s,X^n_s)dW(s) + \int_0^t c(s,X^n_s)dZ^n(s)
\end{equation}
with the initial condition $X^n(s) = \eta(s)$, $s\in[-r,0]$.
\begin{proposition}
Let $X$ and $X^n$ be solutions of \eqref{main-sdde} and \eqref{sdde-n} respectively, and $\norm{Z-Z^n}_{\alpha;[0,T]}\to 0$, $n\to \infty$, in probability. Then $\norm{X-X^n}_T\to 0$, $n\to\infty$, in probability.
\end{proposition}

\end{document}